\documentclass[en, bibtex, reqno, secnum, custombib]{cls/myamspaper}

\title[The isomorphism problem for systems with discrete spectrum]{On the isomorphism problem for non-ergodic 
systems with discrete spectrum}


\newcommand{\X}{\mathrm{X}}
\newcommand{\Y}{\mathrm{Y}}
\renewcommand{\1}{\mathbbm{1}}

\newcommand{\Kro}{\operatorname{Kro}}

\let\ab\allowbreak

\DefineBibliographyStrings{english}{%
  pages = {\hspace{0cm}},
  page = {\hspace{0cm}}
}

\newcommand{\mrL}{\mathrm{L}}

\AtEveryBibitem{%
  \clearfield{issn} 
  \clearfield{doi} 
  \clearfield{number}

  \ifentrytype{online}{}{
    \clearfield{url}
  }
}

\begin{document}

\begin{abstract}
The article presents a new perspective on the isomorphism problem for non-ergodic 
measure-preserving dynamical systems with discrete spectrum which is based on the 
connection between ergodic theory and topological dynamics constituted by topological 
models. By first solving the isomorphism problem for a certain class 
of topological dynamical systems, it is shown that the measure-preserving 
case can in fact be deduced from the topological one via the construction of 
topological models. As a byproduct, a new characterization of 
mean ergodicity for topological dynamical systems is obtained.
\end{abstract}

\maketitle

The isomorphism problem is one of the most important problems in ergodic theory, 
first formulated by von Neumann in
\cite[pp. 592--593]{Neumann1932}, his seminal work
on the Koopman operator method and dynamical systems with ``pure point 
spectrum'' (or \enquote{discrete spectrum}). 
Von Neumann, in particular, asked whether
unitary equivalence of the associated Koopman operators 
(``spectral isomorphy'') implies the existence of a point
isomorphism between two systems (``point isomorphy'').
In \cite[Satz IV.5]{Neumann1932}, he 
showed that two \emph{ergodic} measure-preserving dynamical systems
with discrete spectrum on standard probability spaces are 
point-isomorphic if and only if they are spectrally
isomorphic. These first results on the 
isomorphism problem considerably shaped the ensuing development of 
ergodic theory.
The next step in this direction was the 
Halmos-von Neumann paper \cite{Halmos1942} in which the 
authors gave a more complete solution to the 
isomorphism problem by addressing
three different aspects:
\begin{itemize}
  \item Uniqueness: For which class of dynamical systems is a given 
  isomorphism invariant $\Gamma$ \emph{complete}, meaning that two systems 
  $(\X, \phi)$ and $(\Y, \psi)$ are isomorphic if and only if $\Gamma(\X, \phi) = \Gamma(\Y, \psi)$?
  \item Representation: What are canonical representatives of isomorphy classes
  of dynamical systems?
  \item Realization: Given an isomorphism invariant $\Gamma$, what is the precise class
  of objects that can be realized as $\Gamma(\X, \phi)$ for a dynamical system $(\X, \phi)$?
\end{itemize}
In addition to the uniqueness theorem from \cite{Neumann1932},
their representation theorem showed
that for each isomorphy class
of ergodic dynamical systems with discrete spectrum, there are canonical representatives
given by ergodic 
rotations on compact groups. Moreover, their realization 
theorem established that every (countable) subgroup
of $\T$ can be realized as the point spectrum of the Koopman operator
corresponding to an ergodic system with discrete spectrum. Hence, 
there is, up to isomorphy, a one-to-one correspondence between
(countable) subgroups of $\T$ and (separable) systems with
discrete spectrum.

In the following years, further efforts towards
a solution of the isomorphism problem were made and we refer to 
\cite{Weiss1972}, \cite{Redei2012} and \cite[Chapters 2 -- 4]{Walters1975}
for a detailed account of these developments of which we only mention
the following few:
Nagel and Wolff generalized the 
Halmos-von Neumann theorem to an abstract, operator-theoretic 
statement in \cite{Nagel1972}, Mackey extended it to ergodic 
actions of separable, locally compact groups in \cite{Mackey1964}
and Zimmer took yet another approach, proving a version for 
extensions having relatively discrete spectrum in \cite{Zimmer1976}. 
However, all these results made use of ergodicity assumptions which 
was first justified by von Neumann in \cite[p. 624]{Neumann1932} by 
referring to the possibility of ergodic decomposition. Later, 
Choksi \cite{Choksi1965} showed that the situation is not as simple
as one might hope and it was only in 1981 that Kwiatkowski \cite{Kwiatkowski1981} 
solved the isomorphism problem for non-ergodic systems with discrete spectrum
by using, as proposed by von Neumann, the ergodic decomposition as well
as measure-theoretic methods. We also mention that recently, Austin 
generalized the Mackey-Zimmer theory to non-ergodic systems in \cite{Austin2010}.

It is the purpose of this article to provide an alternative approach to 
the solution of the isomorphism problem for non-ergodic systems with discrete
spectrum: The topological version of the Halmos-von Neumann
theorem has an elegant and well-known proof
using the Ellis (semi)group and Pontryagin duality for compact groups, see 
\cref{hvn} and the introduction of \Cref{sec:hvn}. Knowing this, the measure-preserving version
can be interpreted as a corollary of the topological result by constructing topological models.
In this article, it is shown that this interplay of topological dynamics 
and ergodic theory extends to the non-ergodic case. More precisely, generalizing the proof
sketched above, we first solve the isomorphism problem for non-minimal 
topological dynamical systems subject to a topological restriction and then
obtain the analogous result for non-ergodic measure-preserving systems as a consequence 
by showing that this topological restriction can always be fulfilled when working 
with topological models. On the way, we obtain an interesting characterization
of mean ergodicity for topological dynamical systems in \cref{thm:mergchar}\ref{item:merg2},
asserting that mean ergodicity, a global property, is equivalent to unique ergodicity
of certain subsystems of a topological dynamical system, a local property.

The article is organized as follows: In \Cref{sec:preliminaries}, we fix our
notation and recall basic results about operators and systems with discrete 
spectrum. In \Cref{sec:bundles}, we introduce the notion of a 
bundle of topological
dynamical systems as well as group rotation bundles.
The solution of the isomorphism problem is then broken down into the representation
theorem in \Cref{sec:representation} and the uniqueness and realization results in \Cref{sec:hvn}.
Our general philosophy for both sections, first solving the corresponding problems for 
topological systems and then obtain the same results for measure-preserving systems via 
topological models, was greatly inspired by \cite{Haase2015}.

\subsection*{Acknowledgement} The author is grateful towards Roland Derndinger, Henrik
Kreidler and Rainer Nagel for inspiring discussions and ideas and wants to thank 
Xiangdong Ye for 
bringing the paper \cite{Kwiatkowski1981} to the author's attention, as well as Kari K\"uster and Uwe
Stroinski for valuable corrections.


\section{Preliminaries}\label{sec:preliminaries}

\subsection{Notation} Our notation and terminology are that of \cite{EFHN2015} in general.
 We abbreviate a probability space $(X, \Sigma, \mu)$ by writing  
 $\mathrm{X} \defeq (X, \Sigma, \mu)$ and if $\phi\colon 
X\to X$ is measurable and measure-preserving, we call $(\mathrm{X}, \phi)$ a 
\emph{measure-preserving dynamical system}. For such a system, we define the 
\emph{Koopman operator} $T_\phi\colon\mathrm{L}^p(\mathrm{X}) \to \mathrm{L}^p(\mathrm{X})$,
$1\leq p \leq \infty$, via $T_\phi f \defeq f\circ \phi$ for $f\in\mathrm{L}^p(\X)$. With this definition,
$T_\phi$ is a bounded operator and in fact a \emph{Markov embedding},
i.e., $T_\phi\1 = \1$, $T_\phi'\1 = \1$ and 
$T_\phi|f| = |T_\phi f|$ for all 
$f\in\mathrm{L}^p(\mathrm{X})$. We say that two measure-preserving 
dynamical 
systems $(\mathrm{X}, \phi)$ and $(\mathrm{Y}, \psi)$ are \emph{point isomorphic}
if there exists an essentially invertible, measurable, measure-preserving map $\theta\colon 
X\to Y$ such that $\theta\circ\phi = \psi\circ\theta$. They 
are \emph{Markov isomorphic} if there is an invertible bi-Markov lattice homomorphism
$S\colon\mathrm{L}^1(\Y)\to\mathrm{L}^1(\X)$ such that 
$T_\phi S = ST_\psi$. 
If $S$ is merely a bi-Markov lattice homomorphism, we call $(\mathrm{Y}, \psi)$ a \emph{Markov factor} of 
$(\mathrm{X}, \phi)$.

If $K$ is a compact space (i.e., $K$ is quasi-compact and Hausdorff)
and $\phi\colon K\to K$ is continuous, we call $(K, \phi)$ a \emph{topological
dynamical system} and again define its Koopman operator $T_\phi\colon\mathrm{C}(K)\to \mathrm{C}(K)$ by
$T_\phi f \defeq f\circ \phi$ for all $f\in\mathrm{C}(K)$. We denote 
the space of regular Borel measures on $K$ by $\mathrm{M}(K)$ and 
identify it with the dual of $\mathrm{C}(K)$ via the Riesz-Markov-Kakutani 
representation theorem. We also let
$\mathrm{M}_\phi(K) \defeq \{ \mu\in \mathrm{M}(K) \mid \mu \text{ is } \phi\text{-invariant}\}$
denote the subspace of $\phi$-invariant measures and 
$\mathrm{M}_\phi^1(K)\defeq \{\mu\in\mathrm{M}_\phi(K) \mid \mu \geq 0,\, \langle \1, \mu\rangle = 1\}$
denote the subspace of $\phi$-invariant probability measures.

If $G$ is a compact topological group and $a\in G$ we 
define $\phi_a\colon G\to G$, $\phi_a(g) \defeq ag$ and 
call the dynamical system $(G, \phi_a)$ the \emph{group rotation} with $a$.
We may also abbreviate $(G, \phi_a)$ by writing $(G, a)$. Since the 
Haar measure $\mathrm{m}$ on $G$ is invariant under rotation, 
the rotation can also be considered as a measure-preserving dynamical
system $(G, \mathrm{m}; a)$.

If $T$ is a linear operator on a vector space $E$, we denote by 
\begin{align*}
  A_n[T] \defeq \frac{1}{n}\sum_{k=0}^{n-1} T^k
\end{align*}
its $n$th \emph{Ces\`aro mean} and drop $T$ from the notation if there
is no room for ambiguity. Furthermore, we call
$\fix(T) \defeq \{ x\in E \mid Tx = x\}$ the \emph{fixed space} of $T$. If $F\subset E$
is a $T$-invariant subspace, we set $\fix_F(T) \defeq \fix(T|_F)$. If $(K, \phi)$
is a topological dynamical system, the fixed space $\fix(T_\phi)$ of its Koopman 
operator is a $\mathrm{C}^*$-subalgebra of $\mathrm{C}(K)$. Similarly, if $(\X, \phi)$ is 
a measure-preserving dynamical system, $\fix_{\mathrm{L}^\infty(\X)}(T_\phi)$ is 
a $\mathrm{C}^*$-subalgebra of $\mathrm{L}^\infty(\X)$. By the Gelfand representation 
theorem (cf. \cite[Theorem I.4.4]{Takesaki1979}) there is a compact space $L$ such that 
$\fix_{\mathrm{L}^\infty(\X)}(T_\phi) \cong \mathrm{C}(L)$. The space $L$ is necessarily
extremally disconnected:
 Since $\fix(T_\phi)$ is a closed sublattice of $\mathrm{L}^1(\X)$,
 the representation theorem for AL-spaces (see \cite[Theorem II.8.5]{Schaefer1970}) 
 shows that there is a compact space $M$ and a Borel probability measure 
 $\mu_M$ on $M$ such that
 \begin{align*}
    \mathrm{C}(L) \cong \fix_{\mathrm{L}^\infty(\X)}(T_\phi)
      \cong \mathrm{L}^\infty(M, \mu_M).
 \end{align*}
 But by \cite[Theorem II.9.3]{Schaefer1970}, $\mathrm{C}(L)$ is isomorphic to a dual
 Banach lattice if and only if $L$ is hyperstonean. In particular, $L$ is extremally
 disconnected. This will be crucial for \Cref{measurablebundle}.

\subsection{Operators with discrete spectrum}

We start with a power-bounded operator $T$ on a Banach space $E$ and briefly 
recall the definition of discrete 
spectrum and the Jacobs semigroup generated by $T$ considered first by 
Konrad Jacobs in \cite[Definition III.1]{Jacobs1956}. 

\begin{definition}
  Let $E$ be a Banach space and $T\in\mathscr{L}(E)$ a power-bounded
  operator on $E$.
  \begin{enumerate}[(i)]
    \item The operator $T$ has \emph{discrete spectrum} if its 
    Kronecker space
  \begin{align*}
    \Kro(T)  \defeq  \overline{\lin}\bigcup_{|\lambda| = 1} \ker\left(\lambda\mathrm{I} - T\right).
  \end{align*}
  is all of $E$.
    \item The \emph{Jacobs semigroup} generated by $T$ is
  \begin{align*}
    \mathrm{J}(T) \defeq \overline{\{T^n \mid n\in\N\}}^\mathrm{wot},
  \end{align*}
  where the 
  closure is taken with respect to the weak operator topology and the semigroup
  operation is the composition of operators.
  \end{enumerate}
\end{definition}

The following characterization of an operator having discrete
spectrum can be found in \cite[Theorem, 16.36]{EFHN2015}.

\begin{theorem}\label{thm:operatordiscrspec}
  The following assertions are equivalent.
  \begin{enumerate}[(i)]
    \item $T$ has discrete spectrum.
    \item $\mathrm{J}(T)$ is a weakly/strongly compact group of invertible operators.
    \item The orbit $\{T^n x \mid n\geq 0\}$ is relatively compact and 
    $\inf_{n\geq 0}\|T^nx\| > 0$ for all $0\neq x\in E$.
  \end{enumerate}
\end{theorem}

\begin{remark}\label{rem:ds}
  If $T$ has discrete spectrum, it is mean ergodic and 
  $\mathrm{J}(T)$ is a compact \emph{abelian} group on which the weak and 
  strong operator topology coincide. It is metrizable if $E$ is.
\end{remark}

\subsection{Systems with discrete spectrum} Next, we consider Koopman 
operators corresponding to dynamical systems. See \cite[Chapters 4, 7]{EFHN2015}
for general information.

\begin{definition}
  We say that a measure-preserving dynamical system 
  $(\mathrm{X}, \phi)$ has discrete spectrum if its Koopman operator 
  $T_\phi$ has discrete spectrum on $\mathrm{L}^1(\mathrm{X})$. Similarly,
  we say that a topological dynamical system $(K, \phi)$ has discrete spectrum
  if $T_\phi$ has discrete spectrum as an operator on $\mathrm{C}(K)$.
\end{definition}

\begin{example}\label{groupdiscrspec}
  If $B$ is a compact space, the trivial dynamical system 
  $(B, \id_B)$ has discrete spectrum. Also, if $G$ is a compact group and $a\in G$,
  the measure-preserving dynamical system $(G, \mathrm{m}; a)$
  has discrete spectrum
  and so does the topological dynamical system $(G, a)$.
\end{example}

If $(K, \phi)$ is a topological dynamical system
and $T_\phi\in\mathscr{L}(\mathrm{C}(K))$ has discrete spectrum, the Jacobs semigroup 
$\mathrm{J}(T_\phi)$ is related to the \emph{Ellis semigroup} $\mathrm{E}(K, \phi) \subset K^K$ defined 
as $\mathrm{E}(K, \phi) \defeq \overline{\{\phi^n \mid n\in\N\}}$, see 
\cite[Section 19.3]{EFHN2015}. The following well-known 
result establishes this connection and gives a topological characterization of the operator theoretic
notion of discrete spectrum. 

\begin{proposition}\label{prop:topdiscrspec}
  Let $(K, \phi)$ be a topological dynamical system. For the 
  Koopman operator $T_\phi$, the following assertions are equivalent.
  \begin{enumerate}[(i)]
    \item\label{item:ds1} $T_\phi$ has discrete spectrum.
    \item\label{item:ds2} $\mathrm{J}(T_\phi)$ is a group of Koopman operators.
    \item\label{item:ds3} $\mathrm{E}(K, \phi)$ is a group of equicontinuous transformations on $K$.
    \item\label{item:ds4} $(K, \phi)$ is equicontinuous and invertible.
  \end{enumerate}
  Moreover, if these conditions are fulfilled, the map
  \begin{align*}
    \Phi\colon \mathrm{J}(T_\phi) \to \mathrm{E}(K, \phi), \quad T_\theta \mapsto \theta
  \end{align*}
  is an isomorphism of compact topological groups.
\end{proposition}

The equivalence of \ref{item:ds1} and \ref{item:ds2} follows from \cref{thm:operatordiscrspec}
and \cite[Theorem 4.13]{EFHN2015}. The equivalence of \ref{item:ds2} and \ref{item:ds3}
follows via the canonical isomorphism $\theta \mapsto T_\theta$ and for the equivalence of 
\ref{item:ds3} and \ref{item:ds4} see \cite[Proposition 2.5]{Glasner2007}.

\section{Bundles of dynamical systems}\label{sec:bundles}
Bundles, e.g. in 
differential geometry or algebraic topology, allow to decompose an object
into smaller objects such that the small parts fit together in a 
structured way. This perspective is of interest in our context when dealing with 
dynamical systems which are not ``irreducible'', i.e., not minimal or ergodic.
We therefore introduce the notion of a bundle of topological dynamical 
systems.

\begin{definition}
  A triple $(K, B, p)$ is called a \emph{bundle} if 
  $K$ and $B$ are topological spaces, $B$ is compact and $p\colon K\to B$ is a 
  continuous surjection. The subsets
  $K_b \defeq p^{-1}(b)$ are called the \emph{fibers} of the bundle and 
  if $f\colon K\to S$ is a function into a set $S$, we denote by 
  $f_b$ its restriction to $K_b$.
  A bundle $(K, B, p)$ is called \emph{compact} if $K$ is 
  compact.
  A tuple $(K, B, p; \phi)$ is called a (\emph{compact}) \emph{bundle of
  topological dynamical 
  systems} if $(K, B, p)$ is a compact bundle and $(K, \phi)$ is a topological
  dynamical system such that each fiber $K_b$ is $\phi$-invariant. We denote by 
  $\phi_b$ the restriction of $\phi$ to $K_b$.
\end{definition}

\begin{remark}
  For a dynamical system $(K, \phi)$ and a compact space $B$, a tuple
  $(K, B, p; \phi)$ is a bundle of dynamical systems if and only if $p$ is a factor map
  from $(K, \phi)$ to $(B, \id_B)$. 
\end{remark}

\begin{example}
\begin{enumerate}[(1)]
  \item Let $(K, \phi)$ be a topological dynamical system, $B$ a singleton and $p\colon K\to B$ the 
  unique map from $K$ to $B$. Then $(K, B, p; \phi)$ is a compact bundle of 
  dynamical systems. If $(K, \phi)$ is minimal, this 
  is the only possible choice of $B$. However, the 
  converse is not true as the system $([0,1], x\mapsto x^2)$ demonstrates.

  \item Let $B = [0,1]$, $K = \T\times B$, $\alpha\colon B\to\T$ be continuous and 
  $\phi_\alpha\colon K\to K$,
  $(z, t)\mapsto (\alpha(t)z, t)$ be the associated rotation on the cylinder $K$. Then 
  $(K, B, p_B; \phi_\alpha)$ is a compact bundle of topological dynamical systems. If 
  $\alpha \equiv a$ for some $a\in\T$, the system $(K, \phi)$ is just the product of the torus 
  rotation $(\T, \phi_a)$ and the trivial system $(B, \id_B)$.

  \item More generally, let $(M, \psi)$ be a topological dynamical system and $B$ be a 
  compact space.
  Then the product system $(M\times B, \psi\times\id_B)$ can be viewed as the 
  bundle
  $(M\times B, B, p_B; \psi\times\id_B)$. Bundles of this form are called 
  \emph{trivial}. 
\end{enumerate}
\end{example}

\begin{definition}
  A \emph{bundle morphism} of bundles $(K_1, B_1, p_1)$ and 
  $(K_2, B_2, p_2)$ is a pair $(\Theta, \theta)$ consisting of continuous 
  functions $\Theta\colon K_1 \to K_2$
  and $\theta\colon B_1\to B_2$ such that the following
  diagram commutes:
  \begin{align*}
    \xymatrix{
      K_1 \ar[d]^{p_1} \ar[r]^{\Theta} & K_2 \ar[d]^{p_2} \\
      B_1 \ar[r]^{\theta} & B_2
    }
  \end{align*}
  A \emph{morphism of compact bundles of topological dynamical systems} 
  $(K_1, B_1,\allowbreak p_1; \phi_1)$ and 
  $(K_2, B_2, p_2;\phi_2)$ is a morphism $(\Theta, \theta)$ of the 
  corresponding bundles such that $\Theta$ is, in addition, a morphism of
  topological dynamical systems.
  If $\Theta$ and $\theta$ are homeomorphisms, we call $(\Theta, \theta)$ an 
  \emph{isomorphism}.
\end{definition}

\subsection{Sections} An important tool for capturing structure in bundles relative to
the base space are sections. We recall the following definition.

\begin{definition}
  Let $(K, B, p)$ be a bundle. A function $s\colon B\to K$ is called 
  a \emph{section} of $(K, B, p)$ if 
  $s(b) \in K_b$ for each $b\in B$.
\end{definition}

Although the existence of
sections is guaranteed by the axiom of choice, there may not exist
 \emph{continuous} sections in general.

\begin{example}\label{nosecexample}
  Take $K = B = \T$ and define $p\colon K\to B$ by $p(z) = z^2$. Then this 
  bundle has no continuous section because sections are injective and an 
  injective, continuous function $s\colon\T\to\T$ is necessarily surjective
  and hence a homeomorphism.
\end{example}

Under additional topological conditions, there are 
positive results. The first result due 
to Ernest Michael involves zero-dimensional spaces. Since 
every totally disconnected compact Hausdorff space is zero-dimensional (see 
\cite[Proposition 3.1.7]{ArhTka2008}) we
only state the following special case. It can be found in \cite[Corollary 1.4]{Michael1956}.

\begin{proposition}\label{michael}
  Let $(K, B, p)$ be a compact bundle such that $K$ is metric, 
  $B$ is totally disconnected and $p$ is open. Then there exists a continuous
  section.
\end{proposition}
For not necessarily metric $K$, \cite[Theorem 2.5]{Gleason1958} 
yields the following.

\begin{theorem}\label{gleason}
  Let $(K, B, p)$ be a compact bundle such that $B$ is extremally disconnected.
  Then there exists a continuous section.
\end{theorem}

In case that there is no continuous section, the so-called pullback construction 
allows to construct bigger bundles 
admitting sections.
We repeat this standard construction in our setting.

\begin{construction}[Pullback bundle]
  Let $(K, B, p; \phi)$ be a bundle of topological dynamical systems, $M$ a 
  compact space and 
  $r\colon M\to B$ a continuous surjection. We then define
  \begin{align*}
    r^*K \defeq \left\{ (m, k) \in M\times K \mid r(m) = p(k) \right\}
  \end{align*}
  and denote the restriction of the canonical projection $p_{M}\colon M\times K \to M$
  to $r^*K$ by $\pi_M$ and the restriction of $\id_{M}\times\,\phi$ to $r^*K$ by $r^*\phi$. 
  Then $(r^*K$, $M$, $\pi_M;$ $r^*\phi)$ is a bundle of topological dynamical systems and 
  $(K, \phi)$ is a factor of $(r^*K, r^*\phi)$
  with respect to the projection $\pi_K$ onto the second component. We 
  obtain the following commutative diagram of dynamical systems:
  \begin{align*}
    \xymatrix{
      (r^*K, r^*\phi) \ar[r]^{\pi_K} \ar[d]^{\pi_M} & (K, \phi) \ar[d]^p \\
      (M, \id_M) \ar[r]^r & (B, \id_B)
    }
  \end{align*}
\end{construction}

\begin{example}\label{bundleexample}
  Let $K = B = \T$, define $\phi\colon K\to K$ by 
  setting $\phi(z) = -z$ and $p\colon K\to B$ by setting $p(z) = z^2$. 
  Set $r\colon [0,1]\to \T$, $t\mapsto \e^{2\pi\i t}$. 
  Then the pullback bundle with respect to $r$ is isomorphic to
  the bundle $([0,1]\times\{-1, 1\}, [0,1], p_{[0,1]}; (t, x) \mapsto (t, -x))$. 
\end{example}

\begin{remark}\label{pullbacksection}
  Given a bundle $(K, B, p; \phi)$, the 
  pullback bundle $(p^*K$, $K$, $\pi$; $p^*\phi)$ admits a continuous section: This pullback bundle is
  constructed by gluing to each point in $K$ its fiber and so the map $s\colon K\to p^*K$,
  $k\mapsto (k, k)$ is a canonical continuous section. In particular, every bundle
  of topological dynamical systems is a factor of a bundle admitting a continuous section.
\end{remark}

\begin{remark}\label{pullbackpreserves}
  Properties like minimality and unique ergodicity of each fiber as well as global properties
  such as
  equicontinuity, invertibility and mean 
  ergodicity are preserved under forming pullback bundles. In particular, discrete 
  spectrum is preserved by this construction.
\end{remark}

\subsection{Operator-theoretic aspects of bundles}
The following 
proposition shows that, up to isomorphism, there is a one-to-one 
correspondence between unital $\mathrm{C}^*$-subalgebras
of $\fix(T_\phi)$ and trivial factors $(B, \id_B)$ of the system $(K, \phi)$.

\begin{proposition}\label{prop:factorbundle}
  Let $(K, \phi)$ be a topological dynamical system and $\mathcal{A}$ a unital
  $\mathrm{C}^*$-subalgebra $\mathcal{A}$ of $\fix(T_\phi)$. Then there is an
  (up to isomorphism) unique bundle $(K, B, p;\phi)$ such that 
  $T_p(\mathrm{C}(B)) = \mathcal{A}$ where $T_p$ denotes the Koopman operator
  of $p$.
\end{proposition}
\begin{proof}
  Let $\mathcal{A}$ be a unital $\mathrm{C}^*$-subalgebra of $\fix(T_\phi)$. By
  the Gelfand-Naimark theorem, there is a compact space $B$ such that 
  $\mathcal{A}\cong\mathrm{C}(B)$. The induced 
  $\mathrm{C}^*$-embedding $\mathrm{C}(B)\hookrightarrow\mathrm{C}(K)$ is
  given by a Koopman operator $T_p$ for a continuous map $p\colon K\to B$.
  Because $T_p$ is injective, $p$ is surjective. Moreover, from one obtains
  from the commutativity of the two diagrams
  \begin{align*}
    \xymatrix{
      \mathrm{C}(K) \ar[r]^-{T_\phi} & \mathrm{C}(K) & & 
      K \ar[d]^-p \ar[r]^-\phi   & K \ar[d]^-p  \\
      \mathrm{C}(B) \ar[r]^-{T_{\id_B}} \ar[u]^-{T_p} & \mathrm{C}(B) \ar[u]^-{T_p} & & 
      B \ar[r]^-{\id_B}  & B
    }
  \end{align*}
  that $\phi(K_b) \subset K_b$, so 
  $(K, B, p;\phi)$ is indeed a bundle of topological dynamical systems and 
  $T_p(\mathrm{C}(B)) = \mathcal{A}$ by construction.
  
  Now take two such bundles $(K, B, p; \phi)$ and $(K, B', p';\phi)$ of
  dynamical systems.
  Then $\mathrm{C}(B)\cong\mathcal{A}\cong\mathrm{C}(B')$ and this isomorphism is again 
  given by a Koopman operator
  $T_\theta\colon \mathrm{C}(B)\to\mathrm{C}(B')$ corresponding to a homeomorphism
  $\theta\colon B'\to B$. This yields that $(\id, \theta)$
  is an isomorphism between the two bundles.
\end{proof}

\begin{remark}\label{factororder}
  \cref{prop:factorbundle} allows to order the bundles corresponding to
  a system $(K, \phi)$ by saying that 
  $(K, B_1, p_1; \phi)$ is finer than $(K, B_2, p_2;\phi)$ if $T_{p_1}(\mathrm{C}(B_1))
  \supset T_{p_2}(\mathrm{C}(B_2))$. The term finer is used here because the
  above inclusion induces a surjective map $r\colon B_1\to B_2$. In light of
  \cref{prop:factorbundle}, there is a \emph{maximal trivial factor} of $(K, \phi)$
  associated to the fixed space $\fix(T_\phi)$. We denote this factor
  by $(L_\phi, \id_{L_\phi})$ and the corresponding factor map $q_\phi\colon K\to L_\phi$, 
  but omit $\phi$ from the notation if the context leaves no room for ambiguity.
\end{remark}

After these preparations,
we characterize mean ergodicity via bundles, showing that 
the global notion of mean ergodicity is in fact equivalent to the \enquote{local} notion
of fiberwise unique ergodicity. The elegant proof for implication \ref{item:merg2} $\Rightarrow$ \ref{item:merg1}
presented here was kindly provided by M. Haase.

\begin{theorem}\label{thm:mergchar}
  Let $(K, \phi)$ be a topological dynamical system and $q\colon K\to L$
  the projection onto its maximal trivial factor. Then the 
  following assertions are equivalent.
  \begin{enumerate}[(a)]
    \item\label{item:merg1} The Koopman operator $T_\phi$ is mean ergodic on $\mathrm{C}(K)$.
    \item\label{item:merg2} Each fiber $(K_l, \phi_l)$ is uniquely ergodic.
    \item\label{item:merg3} For each $l\in L$ and $f\in\mathrm{C}(K)$ there is a $c_l\in\C$ such that 
    $A_nf(x) \to c_l$ for all $x\in K_l$.
    \item\label{item:merg4} The map $\mathrm{M}_\phi(K)\to\mathrm{M}(L)$, $\mu \mapsto T_q'\mu = q_*\mu$
    is an isomorphism.
  \end{enumerate}
\end{theorem}
\begin{proof}
  Suppose that $T_\phi$ is mean ergodic. For each $f\in \fix(T_{\phi_l})$ there is 
  a function $\tilde{f}\in \mathrm{C}(K)$ such that 
  $\tilde{f}|_{K_l} = f$. If we denote the mean ergodic projection of $T_\phi$
  by $P$, then $P\tilde{f}\in\fix(T_\phi) = T_q(\mathrm{C}(L))$ and hence $P\tilde{f}$ is
  constant on each fiber. Therefore, $f = P\tilde{f}|_{K_l}$ is constant and $\fix(T_{\phi_l})$
  is one-dimensional. Thus, $T_{\phi_l}$ is mean ergodic since the Ces\'{a}ro averages 
  converge uniformly on $K$ and in particular on $K_l$. Hence, each fiber 
  $(K_l, \phi_l)$ is uniquely ergodic.
  
  Now assume that each fiber $(K_l, \phi_l)$ is uniquely ergodic and let 
  $\mu_l$ denote the corresponding unique invariant probability measure.
  Using this and \cref{supportlem} below, we obtain that the graph of the map
  $l\mapsto \mu_l$ is closed. Since this map takes values in the compact
  set $\mathrm{M}_\phi^1(K)$, applying the closed graph theorem for compact 
  spaces 
  (see \cite[Theorem XI.2.7]{Dugundji1967})
  we conclude that 
  the map $l\mapsto \mu_l$ is weak$^*$-continuous.
  Since each fiber 
  is uniquely ergodic, we also have
  \begin{align*}
    \lim_{n\to\infty} A_nf(x)
    = \int_K f\dmu_{q(x)}
    = \big\langle f, \mu_{q(x)}\big\rangle
  \end{align*}
  and this depends continuously on $x$, showing that $T_\phi$ is mean ergodic.
  The equivalence of \ref{item:merg2} and \ref{item:merg3} is well-known
  for each fiber. Assertion \ref{item:merg4} implies that $\fix(T_\phi)$ separates $\fix(T_\phi')$
  and hence that $T_\phi$ is mean ergodic. Conversely, if $T_\phi$ is mean ergodic, 
  a short calculation shows that the inverse of the map in \ref{item:merg4} is given by 
  $\nu \mapsto \int_L \mu_l\dnu$ where $\mu_l$ is the unique $\phi$-invariant 
  probability measure on $K$.
\end{proof}

\begin{lemma}\label{supportlem}
  Let $(K, \phi)$ be a topological dynamical system and $\mu\in\mathrm{M}(K)$ a probability measure.
  Then $\supp(\mu) \subset K_l$ if and only if $T_q'\mu = \delta_l$.
\end{lemma}
\begin{proof}
  Assume that $\supp(\mu) \subset K_l$. If $g\in \mathrm{C}(L)$ satisfies
  $g(l) = 0$, then $T_qg$ is zero on $K_l$ and hence on $\supp(\mu)$, meaning 
  that 
  \begin{align*}
    \langle g, T_q'\mu\rangle = \langle T_qg, \mu\rangle = 0.
  \end{align*}
  So $\supp(T_q'\mu) \subset \{l\}$ and since $T_q'\mu$ is a probability
  measure, we conclude that $T_q'\mu = \delta_l$.

  Now suppose $T_q'\mu = \delta_l$.
  If $f\in\mathrm{C}(K)$ is positive and such that 
  $\|f\| \leq 1$ and $\supp(f)\cap K_l = \emptyset$, then $l\not\in q(\supp(f))$ and by Urysohn's 
  lemma there is a function $g\in\mathrm{C}(L)$ equal to 1 on 
  $q(\supp(f))$ satisfying $g(l) = 0$. But then $f\leq T_qg$ and hence
  \begin{align*}
    0 
    \leq \langle f, \mu\rangle 
    \leq \langle T_qg, \mu\rangle 
    = \langle g, T_q'\mu\rangle
    = \langle g, \delta_l\rangle
    = 0.
  \end{align*}
  So $\langle f, \mu\rangle = 0$ and we conclude that $\supp(\mu) \subset K_l$
\end{proof}

\begin{remark}\label{mergbasespace}
  In the proof of the implication \ref{item:merg2} $\implies$ \ref{item:merg1}, the information that we
  considered fibers with respect to the maximal trivial factor $L$ was not used. 
  In fact, let $B$ be any trivial factor 
  such that the corresponding
  fibers are uniquely ergodic. The existence of a continuous surjection 
  $r\colon L\to B$ from \cref{factororder} then shows that each fiber $K_b$ is contained in
  a fiber $K_l$. But since each fiber $(K_l, \phi_l)$ is also uniquely ergodic by \cref{thm:mergchar},
  it cannot contain more than one of the sets $K_b$ and so $r$ has to be a homeomorphism.
  Therefore, any bundle of topological dynamical systems $(K, B, p; \phi)$ with uniquely
  ergodic fibers is automatically isomorphic
  to the bundle $(K, L, q; \phi)$ and we may hence assume that $B = L$ and 
  $p = q$.
\end{remark}

\subsection{Group bundles}\label{sec:groupbundles}
We now introduce the main object of this paper: bundles of topological dynamical
systems for which each fiber is a group rotation.

\begin{definition}
  A bundle $(\mathcal{G}, B, p)$ is called a \emph{group bundle} if 
  $(\mathcal{G}, B, p)$ is a bundle and each fiber 
  $\mathcal{G}_b$ carries a group structure such that 
  \begin{enumerate}[(i)]
   \item the multiplication 
   \begin{align*}
      \{(g, g')\in \mathcal{G}\times \mathcal{G} \mid p(g) = p(g')\} \to \mathcal{G}, \quad (g, g') \mapsto gg'
   \end{align*}
   is continuous,
   \item the inverse $\mathcal{G}\to \mathcal{G}$, $g \mapsto g^{-1}$ is continuous and
   \item the neutral element $e_b\in \mathcal{G}_b$ depends continuously on $b\in B$.
  \end{enumerate}
  A bundle $(\mathcal{G}, B, p; \phi)$ is called a \emph{group 
  rotation bundle} if $(\mathcal{G}, B, p)$ is a group bundle, 
  $\phi\colon\mathcal{G}\to\mathcal{G}$ is continuous and
  \begin{enumerate}[(i), resume]
  \item there is a continuous
    section $\alpha\colon B\to \mathcal{G}$ with $(\mathcal{G}_b, \phi_b) = (\mathcal{G}_b, \phi_{\alpha(b)})$.
  \end{enumerate}
  A morphism $(\Theta, \theta)\colon (\mathcal{G}_1, B_1, p_1) \to (\mathcal{G}_2, B_2, p_2)$
  of group bundles is a bundle morphism such that $\Theta$ is a group homomorphism restricted to 
  each fiber. It is called a \emph{morphism of group rotation bundles} if, in addition, 
  $\Theta$ is a morphism of the corresponding dynamical systems.
  A group bundle $(\mathcal{G}, B, p)$ is called \emph{trivial} if 
  there is a group $G$ such that $(\mathcal{G}, B, p) = (G\times B, B, \pi_B)$. We call
  $(\mathcal{G}, B, p)$ \emph{trivializable} if there is an isomorphism
  \begin{align*}
    \iota\colon (\mathcal{G}, B, p) \to (G\times B, B, \pi_B).
  \end{align*}
  We call it \emph{subtrivializable} and $\iota$ a ($G$-)\emph{subtrivialization} if 
  $\iota$ is merely an embedding. We say that two subtrivializations 
  \begin{align*}
    \iota_1\colon (\mathcal{G}_1, B_1, p_1) \to (G\times B_1, B_1, \pi_{B_1}) \\
    \iota_2\colon (\mathcal{G}_2, B_2, p_2) \to (G\times B_2, B_2, \pi_{B_2})
  \end{align*}
  are \emph{isomorphic} if there is an isomorphism 
  $(\Theta, \theta)\colon (\mathcal{G}_1, B_1, p_1) \to (\mathcal{G}_2, B_2, p_2)$
  such that the diagram 
  \begin{align*}
    \xymatrix@C=2.5cm{
      G\times B_1 \ar[r]^-{(g, b) \mapsto (g, \theta(b))} & G\times B_2 \\
      \mathcal{G}_1 \ar[u]^-{\iota_1} \ar[r]^-{\Theta} & \mathcal{G}_2 \ar[u]^-{\iota_2}
    }
  \end{align*}
  commutes.
\end{definition}

\begin{example}\label{nosection}
  As an example of a bundle of topological dynamical systems for which each fiber is a 
  group rotation, yet no continuous section $\alpha\colon B\to K$ exists, recall 
  the bundle from \cref{nosecexample} and equip it with the dynamic 
  $\phi\colon K\to K$, $z\mapsto -z$.
  The fibers here may be interpreted as copies of $(\Z_2, n \mapsto n + 1)$ and 
  it was seen in \cref{nosecexample} that this bundle does not admit continuous
  sections. 
\end{example}

\begin{remark}
  Products and pullbacks of group rotation bundles canonically are again group 
  rotation bundles. 
  However, when passing to factors, the existence of continuous sections may be
  lost, as seen in \cref{bundleexample}. If, however, such a factor $(\mathcal{G}', B', p'; \phi')$ 
  has a continuous section $s\colon B'\to \mathcal{G}'$, it is again a group 
  rotation bundle.
\end{remark}

\begin{remark}
  The notion of group bundles is not new: It has been considered 
  as a special case of locally compact groupoids, in e.g., \cite[Chapter 1]{Renault1980}.
\end{remark}

In order to decompose systems with discrete spectrum, we single 
out group rotation bundles for which each fiber is minimal. Recall the following 
characterization of minimal group rotations.

\begin{theorem}[{\cite[Theorem 10.13]{EFHN2015}}]\label{minimalchar}
  Let $G$ be a compact group, $\mathrm{m}$ the Haar measure on $G$ and consider
  a group rotation $(G, a)$. Then the following assertions are equivalent.
  \begin{enumerate}[(a)]
    \item $(G, a)$ is minimal.
    \item $(G, a)$ is uniquely ergodic.
    \item $\mathrm{m}$ is the only invariant 
    probability measure for $(G, a)$.
    \item $(G, \mathrm{m}; a)$ is ergodic.
  \end{enumerate}
\end{theorem}
\vspace{0.2cm}

\begin{remark}\label{disintegration}
  Let $(\mathcal{G}, B, p; \phi)$ be a group rotation bundle such that each 
  fiber is minimal. Then by \Cref{minimalchar}, 
  every fiber is uniquely ergodic, the unique
  $\phi$-invariant probability measure being the Haar measure $\mathrm{m}_b$ on the
  group $\mathcal{G}_b$. 
  \cref{mergbasespace} yields that we may therefore assume that 
  $B = L$ and $p = q$ where $q\colon \mathcal{G}\to L$ is the projection onto 
  the maximal trivial factor. Moreover, if $\mathrm{m}_l$ denotes the 
  Haar measure on $\mathcal{G}_l$, the map 
  $l \mapsto \mathrm{m}_l$ is weak$^*$-continuous. If $\mu$ is a 
  $\phi$-invariant measure on $\mathcal{G}$, we
  define the pushforward measure $\nu \defeq q_*\mu$ on $L$ and
  disintegrate $\mu$ as in the proof of \cref{thm:mergchar} via
  \begin{align*}
    \mu = \int_{L} \mathrm{m}_l \dnu.
  \end{align*}
  This will be important for \cref{measurablebundle}.
\end{remark}

We now turn towards a generalization of the dual of a locally compact group.
This will be needed for \Cref{sec:hvn}.

\begin{construction}
  Let $(\mathcal{G}, B, q)$ be a locally compact group bundle. Set
  \begin{align*}
    \mathcal{G}^* \defeq \bigcup_{b\in B} \left(\mathcal{G}_b\right)^*
  \end{align*}
  where $(\mathcal{G}_b)^*$ is the dual group of $\mathcal{G}_b$
  and denote by $\pi_B\colon \mathcal{G}^* \to B$ the canonical projection onto
  $B$. Next, let $h\in\mathrm{C}_{\mathrm{c}}(\mathcal{G})$,
  $F\in \mathrm{C}(\mathcal{G})$ and $\epsilon > 0$. Set
  \begin{align*}
    N(h, F, \epsilon) \defeq 
    \left\{ \chi\in\mathcal{G}^* \mmid  
    \left\| \chi h_{\pi_B(\chi)} - 
(Fh)_{\pi_B(\chi)} \right\|_\infty < \epsilon\right\}.
  \end{align*}
  The family of these sets forms a subbasis for a 
  topology which we call
  the \emph{topology of compact convergence} on $\mathcal{G}^*$. 
  
  With this topology, the projection 
  $\pi_B$ is continuous as can be deduced from the continuity of the 
  neutral element section $e\colon B\to \mathcal{G}$ by invoking Urysohn's lemma and Tietze's
  extension theorem to construct appropriate functions $h$ and $F$. 
  Therefore, 
  $(\mathcal{G}^*, B, \pi_B)$ is a bundle which we call the \emph{dual bundle} of 
  $(\mathcal{G}, B, q)$ and hence may also denote by $(\mathcal{G}, B, q)^*$.
  If $(\Theta, \theta) \colon (\mathcal{G}, L, q) \to (\mathcal{H}, 
  L, p)$ is a
  morphism of group bundles such that $\theta$ is bijective, define its \emph{dual 
  morphism} $(\Theta^*, \theta^{-1}) \colon (\mathcal{H}^*, B', q) \to (\mathcal{G}^*, 
  B, p)$ by setting
  $\Theta^*\colon \mathcal{H}^* \to \mathcal{G}^*$,  
  $\chi \mapsto (\Theta_{\pi_L(\chi)})^*\chi$.
\end{construction}

For later reference and the convenience of the reader, we list some 
basic properties of dual bundles. To this end, we recall the following
notions.

\begin{definition}
  Let $X$ and $Y$ be topological spaces. A map $F\colon X\to \mathcal{P}(Y)$ is
  called \emph{lower-} (resp. \emph{upper-}) \emph{semicontinuous} in a point 
  $x\in X$ if for every open $U\subset Y$ such that $F(x)\cap U \neq \emptyset$
  (resp. $F(x) \subset U$) there exists an open neighborhood
  $V$ of $x$ such that for all $x'\in V$ one has $F(x')\cap U \neq \emptyset$
  (resp. $F(x') \subset U$).
  A bundle $(K, B, p)$ is called \emph{lower-} (resp. \emph{upper-}) \emph{semicontinuous} 
  if the map $b\mapsto K_b$ is lower-semicontinuous in each point $b\in B$ and 
  \emph{continuous} if it is both lower- and upper-semicontinuous.
\end{definition}

\begin{remark}
  A bundle is lower-semicontinuous if and only if the bundle projection is an open map.
\end{remark}

\begin{definition}
  Let $(X, B, p)$ and $(Y, B, q)$ be two bundles. Then their \emph{sum}
  is defined as $(X\oplus Y, B, \pi_B)$ where 
  \begin{align*}
    X\oplus Y \defeq \{ (x, y)\in X\times Y \mid p(x) = q(y)\}
  \end{align*}
  and $\pi_B(x, y) \defeq p(x)$.
\end{definition}

\begin{proposition}\label{dualprops}
  Let $(\mathcal{G}, B, p)$ and $(\mathcal{H}, B', p')$ be 
  locally compact abelian group bundles and 
  $(\Theta, \theta)\colon (\mathcal{G}, B, p) \to (\mathcal{H}, B', p')$
  a morphism of group bundles such that $\theta$ is bijective.
  \begin{enumerate}[(i)]
    \item \label{item:first} The evaluation map
      $\mathrm{ev}\colon\mathcal{G}^*\oplus\mathcal{G} \to \C$, $(\chi, g) \mapsto \chi(g)$
    is continuous. In fact, a net $(\chi_{i})_{i\in I}$ converges to $\chi\in\mathcal{G}^*$ if 
    and only if 
    $\pi_B(\chi_i)\to \pi_B(\chi)$ and for every convergent net $(g_i)_{i\in I}$ with
    $p(g_i) = \pi_B(\chi_i)$ and limit 
    $g\in \mathcal{G}$ we have $\chi_{i}(g_{i}) \to \chi(g)$.
    \item \label{item:unambiguous} For $b\in B$, $\id_{(\mathcal{G}_b)^*}\colon
    (\mathcal{G}_b)^*\to(\mathcal{G}^*)_b $ is an isomorphism of locally compact groups. 
    In particular, the notation $\mathcal{G}_b^*$ is unambiguous.
    \item \label{item:bundleproduct} If $G$ is a locally compact group and $L$ is a compact space,
    $(G\times L, L, \pi_L)^* = (G^*\times L, L, \pi_L)$.    
    \item \label{item:bundlehausdorff} The bundle $(\mathcal{G}, B, p)$ is lower-semicontinuous if 
    and only if $\mathcal{G}^*$ is a 
    Hausdorff space.
    \item \label{item:dualcont}The dual morphism $\Theta^*$ is continuous and
    \begin{align*}
      (\Theta^*, \theta^{-1}) \colon (\mathcal{H}^*, B', \pi_{B'}) \to (\mathcal{G}^*, B, \pi_B)
    \end{align*}
    is a morphism of group bundles.
    \item\label{item:properembed}
    If $\Theta$ is proper and surjective, $\Theta^*\colon\mathcal{H}^*\to \mathcal{G}^*$
    is an embedding.
    \item\label{item:compactembed} If $\mathcal{G}$ is compact and $\Theta$ surjective, $\Theta^*\colon\mathcal{H}^*\to 
\mathcal{G}^*$
    is an embedding.
  \end{enumerate}
\end{proposition}
\begin{proof}
  The first part of \ref{item:first} follows from the definition of the 
  topology on $\mathcal{G}^*$ using local compactness to invoke Urysohn's lemma
  and Tietze's extension theorem which provide appropriate functions $h$ and $F$.
  The second part of \ref{item:first} is a simple proof by contradiction.
  For part \ref{item:unambiguous}, it suffices to show that the two sets carry
  the same topology. This follows from \ref{item:first} since it shows that the two topologies 
  have the same convergent nets. By the same argument, \ref{item:bundleproduct} follows 
  directly from \ref{item:first} and so does \ref{item:dualcont}, since 
  it suffices to show that 
  $\Theta^*$ is continuous. In
  \ref{item:bundlehausdorff}, we obtain the Hausdorff property from lower-semicontinuity 
  and \ref{item:first}, showing that every convergent net in $\mathcal{G}^*$ has a unique limit.
  For the converse implication in \ref{item:bundlehausdorff}, assume that $\mathcal{G}^*$ is Hausdorff but 
  $(\mathcal{G}, B, p)$ is not lower-semicontinuous. Then there exist $b\in B$, 
  an open subset $U\subset \mathcal{G}$ with $U\cap\mathcal{G}_b \neq \emptyset$ and a 
  net $(b_i)_{i\in I}$ such that $b_i\to b$ and $\mathcal{G}_{b_i}\cap U = \emptyset$ for all $i\in I$.
  Consider
  \begin{align*}
    H \defeq \left\{g \in \mathcal{G}_b \mmid \begin{matrix}
                                \text{There exists a net } (g_i)_{i\in I} \text{ in } \mathcal{G} \\
                                \text{ with } g_i \to g \text{ and } g_i\in\mathcal{G}_{b_i} \text{ for all } i\in I
                              \end{matrix}\right\}.
  \end{align*}
  This set forms a proper subgroup of $\mathcal{G}_b$ which is not dense 
  since $H\cap U_b = \emptyset$.
  If we now denote by $\chi_0^{(b_i)}$ the trivial 
  character on $\mathcal{G}_{b_i}$, then $\chi_0^{(b_i)} \to \chi$ for any character 
  $\chi\in\mathcal{G}_b^*$ such that $\chi|_{\overline{H}} \equiv 1$. But since 
  $\overline{H}$ is a proper subgroup of $\mathcal{G}_b$,
  there are at least two characters satisfying this. In particular, 
  $\mathcal{G}^*$ cannot be Hausdorff. 
  
  For part \ref{item:properembed} (which trivially implies \ref{item:compactembed}), note that 
  $\Theta^*$ is injective because $\Theta$ is surjective. Let 
  $(\chi_i)_{i\in I}$ be a net in $\mathcal{H}^*$ such that 
  $\Phi^*(\chi_i)$ converges to $\eta\in\mathcal{G}^*_b$. Then 
  $\eta(g) = \eta(g')$ if $\Theta(g) = \Theta(g')$ and so $\eta = \chi\circ \Theta$
  for a function $\chi\colon \mathcal{H}_b\to\C$. It is again multiplicative and continuous 
  because $H_b$ carries the final topology with respect to $\Theta_b$, so 
  $\chi\in\mathcal{H}_b^*$. Let $N(U, h, F, \epsilon)$ be an open neighborhood
  of $\chi$. Then $N(\theta^{-1}(U), h\circ \Theta, F\circ\Theta, \epsilon)$ is an open
  neighborhood of $\chi\circ\Theta$ and so $\chi_i\circ\Theta\in 
  N(\theta^{-1}(U), h\circ \Theta, F\circ\Theta, \epsilon)$ for $i \geq i_0$, implying 
  $\chi_i\in N(U, h, F, \epsilon)$ for $i \geq i_0$. Hence, $\chi_i\to \chi$.
\end{proof}


\section{Representation}\label{sec:representation}

The classical examples for systems with discrete spectrum are group 
rotations $(G, a)$ and trivial systems 
$(B, \id_B)$ as seen in \Cref{groupdiscrspec}. In \cref{bundlefactor} we show that, in fact, every system
with discrete spectrum canonically is a factor of a product 
$(G, a)\times (B, \id_B)$ and therefore can be constructed from these
two basic systems. We start with the topological case
and derive the measure-preserving case from this via topological models. As a result,
we generalize the Halmos-von Neumann representation theorem to not necessarily
minimal or ergodic systems with discrete spectrum in \cref{measurablebundle}.

We briefly recall the Halmos-von Neumann theorem for minimal topological systems $(K, \phi)$
and, because the 
proof of \cref{bundlechar} below is based on it, sketch a proof using the Ellis (semi)group
$\mathrm{E}(K, \phi) \defeq \overline{\{\phi^k \mid k\in\N\}}\subset K^K$
introduced by Ellis as the \enquote{enveloping semigroup}, see \cite{Ellis1959}.

\begin{theorem}\label{hvn}
  Let $(K, \phi)$ be a minimal topological dynamical system with discrete spectrum. Then 
  $(K, \phi)$ is isomorphic to a minimal group rotation $(G, \phi_a)$ on 
  an abelian compact group $G$. More precisely, for each $x_0\in K$ there 
  is a unique isomorphism $\delta_{x_0}\colon (\mathrm{E}(K, \phi), \phi) \to (K, \phi)$
  such that $\delta_{x_0}(\id_K) = x_0$.
\end{theorem}
\begin{proof}
  Pick a point $x_0\in K$ and consider the map
  \begin{align*}
    \delta_{x_0}\colon \mathrm{E}(K, \phi)\to K, \quad \psi \mapsto \psi(x_0).
  \end{align*}
  Since $K$ is minimal, $\delta_{x_0}$ is injective.
  Moreover, $\delta_{x_0}(\mathrm{E}(K, \phi))$ is a closed, invariant subset of $K$ which is not empty and
  hence $\delta_{x_0}(\mathrm{E}(K, \phi)) = K$. It is not difficult to check that the system
  $(\mathrm{E}(K, \phi), \phi)$ is isomorphic to $(K, \phi)$ via $\delta_{x_0}$.
\end{proof}

Note that the isomorphism in \Cref{hvn}
depends on the (non-canonical) choice of $x_0\in K$.

\begin{definition}\label{ellisbundle}
 Let $(K, B, p; \phi)$ be a bundle of topological dynamical systems. We then 
 set 
 \begin{align*}
  &\mathrm{E}(K, B, p; \phi) \defeq \bigcup_{b\in B} \mathrm{E}(K_b, \phi_b), \\
  &\alpha\colon B \to \mathrm{E}(K, B, p; \phi), \quad b\mapsto \phi_b, \\
  &\phi_\alpha\colon \mathrm{E}(K, B, p; \phi)  \to \mathrm{E}(K, B, p; \phi), \quad 
  \psi_b \mapsto \psi_b\circ\phi_b
 \end{align*}
 and let $\pi_B\colon\mathrm{E}(K, B, p; \phi) \to B$ denote the projection 
 onto the second component. We equip $\mathrm{E}(K, B, p; \phi)$ with the 
 final topology induced by the map $\rho\colon 
 \mathrm{E}(K, \phi)\times B\to \mathrm{E}(K, B, p; \phi)$, $(\psi, b)\mapsto \psi_b$
 and call $(\mathrm{E}(K, B, p; \phi), B, \pi_B; \alpha)$
 the \emph{Ellis semigroup bundle} of $(K, B, p; \phi)$.
\end{definition}

We abbreviate the Ellis semigroup bundle by $\mathrm{E}(K, B, p; \phi)$ if the 
context leaves no room for ambiguity.
We also note that it is a group rotation bundle if it is compact and 
$\mathrm{E}(K, \phi)$ is a group, in which case we call it the \emph{Ellis group
bundle}. We now give a criterion for the space $\mathrm{E}(K, B, p; \phi)$ to be compact.


\begin{lemma}\label{quotientcompact}
  Let $X$ be a locally compact space, $Y$ a (Hausdorff) uniform space,
  $B$ a compact space and $F\colon B \to\mathcal{P}(X)$ a set-valued map.
  Define an equivalence relation $\sim_F$ on $\mathrm{C}(X, Y) \times B$ via
  \begin{align*}
    (f, b) \sim_F (g, b') \quad \text{if} \quad 
    b = b' \text{ and } f|_{F(b)} = g|_{F(b)}
  \end{align*}
  and endow $\mathrm{C}(X, Y)$ with the topology of locally uniform convergence.
  Moreover, let $A\subset \mathrm{C}(X, Y)$ be a compact subset.
  If $F$ is lower-semicontinuous, then the 
  quotient $A\times B/{\sim_F}$ is a compact space.
\end{lemma}
\begin{proof}
 Since the quotient of a compact space by a closed equivalence relation is again
 compact (cf. \cite[Proposition 10.4.8]{BourbakiGenTop1}), it suffices to 
 show that $\sim_F$ is closed. So let $((f_i, b_i), (g_i, b_i))_{i\in I}$ be 
 a net in $\sim_F$ with limit $((f, b), (g, b)) \in (\mathrm{C}(X, Y)\times B)^2$.
 Pick $x\in F(b)$. Since $F$ is lower-semicontinuous and $b_i \to b$, there is 
 a net $(x_i)_{i\in I}$ such that $x_i\in F(b_i)$ and $x_i\to x$. But since 
 $(f_i)_{i\in I}$ and $(g_i)_{i\in I}$ converge locally uniformly, 
 \begin{align*}
   f(x) 
   = \lim_{i\to\infty} f_i(x_i)
   = \lim_{i\to\infty} g_i(x_i)
   = g(x).
 \end{align*}
 Since $x\in F(b)$ was arbitrary, it follows that $f|_{F(b)} = g|_{F(b)}$ and
 so $\sim_F$ is closed.
\end{proof}

%

\begin{lemma}\label{lem:open}
 Let $(K, B, p; \phi)$ be a bundle of topological dynamical systems such
 that $(K, \phi)$ is equicontinuous. 
 Then the following assertions are true.
 \begin{enumerate}[(i)]
  \item If $p$ is open, the space $\mathrm{E}(K, B, p; \phi)$ is compact.
  \item If each fiber $(K_b, \phi_b)$ is minimal, then $p$ is open.
 \end{enumerate}
\end{lemma}
\begin{proof}
 Part (i) is a special case of \Cref{quotientcompact} with 
 $A = \mathrm{E}(K, \phi)$ and $F\colon B\to \mathcal{P}(K)$ with $F(b) = K_b$ 
 since the topologies of pointwise and uniform
 convergence coincide on equicontinuous subsets of $\mathrm{C}(K, K)$.
 
 For (ii), assume that each fiber $(K_b, \phi_b)$ is minimal. If 
 $U\subset K_b$ is open in $K_b$, then
 \begin{align*}
  K_b = \bigcup_{k = 0}^\infty \phi_b^{-k}(U)
 \end{align*}
 since $(K_b, \phi_b)$ is minimal (cf. \cite[Proposition 3.3]{EFHN2015}). 
 Therefore, if $U\subset K$ is open, then 
 \begin{align*}
  p^{-1}(p(U)) = \bigcup_{k=0}^\infty \phi^{-k}(U)
 \end{align*}
 is open in $K$ and hence $p(U)$ is open.
\end{proof}

\begin{proposition}\label{fibersminimal}
  Let $(K, \phi)$ be a topological dynamical system with discrete spectrum
  and $q\colon K\to L$ the canonical projection onto the maximal trivial factor. 
  Then each fiber 
  $(K_l, \phi_l)$ is minimal and has discrete spectrum.
\end{proposition}
\begin{proof}
  Each fiber $(K_l, \phi_l)$ has discrete 
  spectrum since $\mathrm{E}(K_l, \phi_l) = \{\psi|_{K_l} \mid \psi\in \mathrm{E}(K, \phi)\}$,
  use \cref{prop:topdiscrspec}.
  Moreover, for $x, y\in K_l$ one has $\overline{\orb}(x) = \mathrm{E}(K_l, \phi_l)x$
  and $\overline{\orb}(y) = \mathrm{E}(K_l, \phi_l)y$. Since $\mathrm{E}(K_l, \phi_l)$
  is a group, we conclude that either $\overline{\orb}(x) = \overline{\orb}(y)$
  or $\overline{\orb}(x)\cap \overline{\orb}(y) = \emptyset$. 
  However, by \cref{rem:ds} the system $(K, \phi)$
  is mean ergodic and hence
  $(K_l, \phi_l)$ is uniquely
  ergodic by \cref{thm:mergchar}. We now conclude from the Krylov-Bogoljubov Theorem 
  (cf. \cite[Theorem 10.2]{EFHN2015}) that $K_l$ cannot
  contain two disjoint closed orbits. Consequently, $\overline{\orb(x)} 
  = \overline{\orb(y)}$ for all $x, y\in K_l$ and hence $(K_l, \phi_l)$ is minimal.
\end{proof}

\begin{theorem}\label{bundlechar}
  Let $(K, \phi)$ be a topological dynamical system with discrete spectrum
  and assume that the canonical projection $q\colon K\to L$ onto the maximal
  trivial factor admits a continuous section. 
  Then $(K, L, q; \phi)$ is isomorphic to its Ellis group bundle.
\end{theorem}
\begin{proof}
  Let $s\colon L\to K$ be a section for $q$.
  By \cref{fibersminimal}, every fiber $(K_l, \phi_l)$ is minimal and has discrete spectrum.
  By \cref{hvn} we obtain an isomorphism 
  $\Phi_l\colon (\mathrm{E}(K_l, \phi_l), \phi_l) \to (K_l, \phi_l)$ 
  satisfying $\Phi_l(\id_{K_l}) = s(l)$. This yields a bijection 
  \begin{align*}
    \Phi\colon \mathrm{E}(K, L, q; \phi) \to K, \quad
    \psi_l \mapsto \psi_l(s(l)).
  \end{align*}
  Because $(K, \phi)$ has discrete spectrum, the map
  \begin{align*}
    \mathrm{E}(K, \phi)\times L \to K, \quad (\psi, l) \mapsto \psi(s(l))  
  \end{align*}
  is continuous, hence $\Phi$ is continuous and an
  isomorphism of topological dynamical systems.
\end{proof}

\cref{nosection} shows that there are systems with discrete spectrum which 
are not isomorphic to a group rotation bundle. However, the following is still true.

\begin{corollary}\label{bundlefactor}
  Let $(K, \phi)$ be a topological dynamical system with discrete spectrum. 
  Then $(K, \phi)$ is a factor of a trivial group rotation bundle $(G, a)\times(B, \id_B)$
  where the group rotation $(G, a)$ is minimal and can be taken
  as $(G,a) = (\mathrm{E}(K, \phi),\phi)$.
\end{corollary}
\begin{proof}
  Let $(K, \phi)$ be a topological dynamical system with discrete spectrum
   and $q\colon K\to L$ the projection onto its maximal trivial factor.
   As noted in \cref{pullbackpreserves}, the 
   associated pullback system
   $(q^*K, K, \pi_K, q^*\phi)$ also has discrete spectrum. Moreover, its fibers are
   uniquely ergodic and so \cref{mergbasespace} shows that
   its maximal trivial factor is homeomorphic to $K$. This, combined with 
   \cref{pullbacksection} yields that the canonical projection onto
   its maximal trivial factor admits a continuous section $s\colon K\to q^*K$.
   By \cref{bundlechar} we obtain that the bundle 
   $(q^*K, K, \pi_K; q^*\phi)$ is isomorphic to its Ellis group bundle which is, 
   by construction, a factor of the trivial group rotation bundle 
   $(\mathrm{E}(q^*K, q^*\phi), q^*\phi)\times(K, \id_K)$. We now consider the 
   following maps:
   \begin{align*}
     &Q\colon \mathrm{E}(K, \phi) \to \mathrm{E}(q^*K, q^*\phi), \quad
     \psi \mapsto q^*\psi, \\
     &P\colon \mathrm{E}(q^*K, q^*\phi) \to \mathrm{E}(K, \phi), \quad
     \tilde{\psi} \mapsto p_2\circ\tilde{\psi}\circ s
   \end{align*}
   where $p_2\colon q^*K\to K$ denotes the projection onto the second component.
   Both $Q$ and $P$ are continuous and satisfy 
   $Q(\phi^k) = (q^*\phi)^k$ and $P((q^*\phi)^k) = \phi^k$ for all 
   $k\in \N$. Since 
   $\phi$ and $q^*\phi$ generate their respective Ellis groups, $P$ and $Q$ are mutually inverse. Hence, 
   \begin{align*}
     (\mathrm{E}(q^*K, q^*\phi), q^*\phi)\times(K, \id_K) 
     \cong (\mathrm{E}(K, \phi), \phi)\times(K, \id_K).
   \end{align*}
\end{proof}

\begin{remark}
  The group rotation $(\mathrm{E}(K, \phi), \phi)$ is the smallest group
  rotation that can be taken as $(G, a)$ in \cref{bundlefactor} in the sense that 
  any such group rotation $(G, a)$ admits an epimorphism $\eta\colon 
  (G, a) \to (\mathrm{E}(K, \phi), \phi)$. This is true because 
  a factor map $\theta\colon (G, a)\times(B, \id_B)\to (K, \phi)$ induces a continuous, surjective
  group homomorphism
  \begin{align*}
     \mathrm{E}(\theta)\colon \mathrm{E}((G, a)\times(B, \id_B))\to \mathrm{E}(K, \phi)
  \end{align*}
  satisfying $\mathrm{E}(\theta)(a\times\id_B) = \phi$ and 
  \begin{align*}
    (\mathrm{E}((G, a)\times(B, \id_B)), a\times\id_B)\cong (\mathrm{E}(G, a), \phi_a) \cong (G, a).
  \end{align*}
\end{remark}

\begin{remark}\label{rem:factor}
  If $(K, \phi)$ has discrete spectrum and the canonical projection $q\colon K\to L$ 
  admits a continuous section, the system is already isomorphic to its Ellis group 
  bundle and hence, by definition of the latter, a factor of the system 
  $(\mathrm{E}(K, \phi), \phi)\times(L,\id_{L})$. In this case, one can take
  $B = L$ in \cref{bundlefactor}.
\end{remark}

%


\subsection{The measure-preserving case}

Since the problem of finding continuous sections can be solved for topological models of 
measure spaces as we will see below, we obtain a
better result for measure-preserving systems.
This is our generalization of the Halmos-von Neumann theorem to the non-ergodic case.
It is proved by constructing a topological model and then applying \Cref{bundlechar}.
For background information on topological models, see
\cite[Chapter 12]{EFHN2015}.

\begin{theorem}\label{measurablebundle}
  Let $(\mathrm{X}, \phi)$ be a measure-preserving system with discrete spectrum. Then 
  $(\mathrm{X}, \phi)$ is Markov-isomorphic to the rotation on a compact group rotation
  bundle. More precisely, there are a compact group rotation bundle 
  $(\mathcal{G}, B, p; \phi_\alpha)$ with minimal fibers and a $\phi_\alpha$-invariant measure
  $\mu_{\mathcal{G}}$ on $\mathcal{G}$ such that 
  $(\mathrm{X}, \phi)$ and 
  $(\mathcal{G}, \mu_{\mathcal{G}}; \phi_\alpha)$ are Markov-isomorphic. Moreover,
  this group rotation bundle can be chosen such that the canonical map 
  $j\colon\Kro_{\mathrm{C}(\mathcal{G})}(T_{\phi_\alpha}) \to 
  \Kro_{\mathrm{L}^\infty(\mathcal{G}, \mu_{\mathcal{G}})}(T_{\phi_\alpha})$
  of Kronecker spaces is an isomorphism.
\end{theorem}\begin{proof}
  We define
  \begin{align*}
   \mathcal{A} \defeq \operatorname{cl}_{\mathrm{L}^\infty}\bigcup_{|\lambda| = 1} 
   \ker_{\mathrm{L}^\infty}(\lambda I - T_\phi)
  \end{align*}
  and note that this is a $T_\phi$-invariant, unital $\mathrm{C}^*$-subalgebra 
  of $\mathrm{L}^\infty(\X)$ which is dense in $\mathrm{L}^1(\X)$ by \cite[Lemma 17.3]{EFHN2015}
  since $(\X, \phi)$ has discrete spectrum.
  The Gelfand representation 
  theorem (cf. \cite[Theorem I.4.4]{Takesaki1979}) yields that there is a compact space $K$ and 
  a $\mathrm{C}^*$-isomorphism $S\colon \mathrm{C}(K) \to \mathcal{A}$. 
  The Riesz-Markov-Kakutani representation theorem shows that there is a unique
  Borel probability measure $\mu_K$ on $K$ such that 
  \begin{align*}
     \int_K f\dmu_K = \int_\X S f \dmu_\X \quad \text{for all } f\in\mathrm{C}(K).
  \end{align*}
  Moreover, $T \defeq S^{-1}\circ T_\phi\circ S\colon \mathrm{C}(K)\to\mathrm{C}(K)$ 
  defines a $\mathrm{C}^*$-homomorphism and so (cf. \cite[Theorem 4.13]{EFHN2015}) there 
  is a continuous map 
  $\psi\colon K\to K$ such that $T = T_\psi$. The operator $S$ is, by construction,
  an $\mathrm{L}^1$-isometry and $S |f| = |S f|$ for all $f\in\mathrm{C}(K)$
  by \cite[Theorem 7.23]{EFHN2015}. Since $\mathcal{A}$ is dense in 
  $\mathrm{L}^1(\X)$, we conclude that $S$ extends to a bi-Markov lattice 
  homomorphism $S\colon\mathrm{L}^1(K, \mu_K) \to \mathrm{L}^1(\X)$.

  The (topological) system $(K, \psi)$ now has discrete spectrum by construction.
  Let $L_\psi$ denote the maximal trivial factor of $(K, \psi)$. Then
  $\mathrm{C}(L_\psi) \cong \fix(T_\psi) 
\cong \fix_{\mathrm{L}^\infty(\X)}(T_\phi)$ and so 
  $L_\psi$ is extremally disconnected as noted in \Cref{sec:preliminaries}. From \Cref{gleason} we therefore
  conclude that the canonical projection $q\colon K\to L_\psi$ has a continuous 
  section. \Cref{bundlechar} shows that there is an isomorphism
  $\theta\colon (K, \psi)\to (\mathcal{G}, \alpha)$ where $(\mathcal{G}, \alpha)$ is the rotation on 
  some compact group rotation bundle with minimal fibers.
  Equipping $(\mathcal{G}, \alpha)$ with the push-forward measure $\mu_\mathcal{G}\defeq\theta_*\mu_K$,
  we obtain that the system $(\X, \phi)$ is isomorphic to the system 
  $(\mathcal{G}, \mu_\mathcal{G}; \alpha)$.
\end{proof}


%

\begin{corollary}\label{markovfactor}
  Let $(\mathrm{X}, \phi)$ be a measure-preserving dynamical system 
  with discrete spectrum and $(L, \nu; \id_L)$ a topological model for 
  $\fix_{\mrL^\infty(\X)}(T_\phi)$. Then $(\mathrm{X}, \phi)$ is a 
  Markov factor of the trivial group rotation bundle 
  $(\mathrm{J}(T_\phi), \mathrm{m}; T_\phi)\times (L, \nu; \id_L)$.
\end{corollary}
\begin{proof}
  This follows from \cref{measurablebundle} and \cref{rem:factor}.
\end{proof}

\begin{remark}
  We can also interpret the Halmos-von Neumann theorem in the following
  way: If $(\mathrm{X}, \phi)$ is an ergodic, measure-preserving system with 
  discrete spectrum, there is a compact, ergodic group rotation $(G, a)$ and 
  a Markov isomorphism $S\colon\mathrm{L}^1(\X) \to \mrL^1(G, \mathrm{m})$
  such that the diagram 
  \begin{align*}
    \xymatrix{
      \mrL^1(\X) \ar[d]^-{T_\phi} \ar[r]^-S & \mrL^1(G, \mathrm{m}) \ar[d]^-{T_{\phi_a}} \\
      \mrL^1(\X)  \ar[r]^-S & \mrL^1(G, \mathrm{m}) 
    }
  \end{align*}
  commutes, i.e., $T_\phi$ acts like an ergodic rotation on scalar-valued functions. If $(\mathrm{X}, \phi)$
  is not ergodic, we can interpret \cref{markovfactor} similarly:
  There is a compact, ergodic group rotation 
  $(G, a)$, a compact probability space $(L, \nu)$ and a Markov embedding $S\colon 
  \mrL^1(\X) \to \mrL^1(G\times L, \mathrm{m}\times\nu)$ such that $T_{\phi_a\times \id_L}S = 
  ST_\phi$.
  The rotation $\phi_a$ induces a Koopman operator $T_{\phi_a}$ on 
  the vector-valued functions in 
  $\mathrm{L}^1(G, \mathrm{m}; \mrL^1(L, \nu))$. With the $\pi$-tensor product 
  (see \cite[Th\'{e}or\`{e}me 2]{Grothendieck1952}), we obtain
  \begin{align*}
    \mathrm{L}^1(G, \mathrm{m}; \mrL^1(L, \nu))
    \cong \mathrm{L}^1(G, \mathrm{m})\otimes\mathrm{L}^1(L, \nu) 
    \cong \mathrm{L}^1(G\times L, \mathrm{m}\times\nu).
  \end{align*}
  Now, the diagram
  \begin{align*}
    \xymatrix{
      \mrL^1(\X;\C)\, \ar[d]^-{T_\phi} \ar@{^(->}[r]^-S & \mrL^1(G\times L, \mathrm{m}\times\nu) \ar[d]^-{T_{\phi_a\times 
\id_L}} \ar[r]^-\cong & \mrL^1(G, \mathrm{m}; \mrL^1(L, \nu)) \ar[d]^-{T_{\phi_a}}  \\
      \mrL^1(\X;\C)\, \ar@{^(->}[r]^-S & \mrL^1(G\times L, \mathrm{m}\times\nu) \ar[r]^-\cong &  \mrL^1(G, \mathrm{m}; 
\mrL^1(L, \nu)) 
    }
  \end{align*}
  also commutes, i.e., $T_\phi$ acts like an ergodic rotation on vector-valued functions. 
  We can 
  interpret the topological Halmos-von Neumann theorem \cref{hvn} and \cref{bundlefactor} analogously.
\end{remark}

\section{Realization and Uniqueness}\label{sec:hvn}


The topological Halmos-von Neumann theorem shows that every minimal dynamical system with discrete spectrum
is isomorphic to a minimal group rotation $(G,a)$. Therefore, minimal group rotations
can be seen as the canonical representatives of minimal systems with discrete spectrum. Moreover, the Pontryagin 
duality theorem
shows that $(G, a)$ and $(G^{**}, \delta_a)$ are isomorphic which has two consequences: On the one hand, 
$G^*\cong G^*(a)$ via $\chi \mapsto \chi(a)$ and $G^*(a) = 
\sigma_{\mathrm{p}}(T_{\phi_a})$
where $T_{\phi_a}$ denotes the Koopman operator of $\phi_a$,
see \cite[Propositions 14.22 and 14.24]{EFHN2015}. In particular, $\sigma_{\mathrm{p}}(T_{\phi_a})$ is a subgroup of 
$\T$ and for the canonical inclusion $\iota\colon \sigma_{\mathrm{p}}(T_{\phi_a})\hookrightarrow \T$
\begin{align*}
  (G, a) \cong (G^*(a)^*, \iota) = (\sigma_\mathrm{p}(T_{\phi_a})^*, \iota)
\end{align*}
if $\sigma_{\mathrm{p}}(T_\phi)$ is endowed with the discrete topology.
Therefore, the point spectrum $\sigma_\mathrm{p}(T_{\phi_a})$ is a complete isomorphism invariant
for the minimal group rotation $(G, a)$. Combined with the Halmos-von Neumann theorem, this shows that 
the point spectrum $\sigma_{\mathrm{p}}(T_\phi)$ is a complete isomorphism invariant for \emph{all} minimal
topological dynamical systems $(K, \phi)$ with discrete spectrum. On the other hand, the Pontryagin duality theorem
also implies that every subgroup of $\T$ can be realized as $\sigma_{\mathrm{p}}(T_{\phi_a})$ for some 
group rotation $(G, a)$. This completes the picture, showing that minimal systems with discrete
spectrum are, up to isomorphism, in one-to-one correspondence with subgroups of $\T$.

In order to generalize these results to the non-minimal setting, we need to adapt the 
Pontryagin duality theorem to group rotation bundles using the preparations from 
\Cref{sec:groupbundles}. We start with the necessary
terminology.

\begin{construction}[Dual bundles]\label{dualbundles}
  If $(\mathcal{G}, L, q; \alpha)$ is a compact group rotation bundle with minimal
  fibers and discrete spectrum, the map 
  \begin{align*}
    \rho\colon \mathrm{E}(\mathcal{G}, \phi_\alpha)\times L \to \mathcal{G}, \quad
    (\psi, l) \mapsto \psi(e_l)
  \end{align*}
  yields a surjective morphism $(\rho, \id_L)$ of group bundles which induces, 
  by \cref{dualprops}, an embedding $\rho^*\colon \mathcal{G}^*\hookrightarrow 
  \mathrm{E}(\mathcal{G}, \phi_\alpha)^*\times L$. Since $\mathrm{E}(\mathcal{G}, \phi_\alpha)$
  is compact, its dual group is discrete and so we also have the embedding
  \begin{align*}
    j\colon \mathrm{E}(\mathcal{G}, \phi_\alpha)^*\times L \to \T\times L, \quad 
    (\chi, l) \mapsto (\chi(\phi_\alpha), l)
  \end{align*}
  where $\T$ carries the discrete topology.
  The composition $\iota\colon \mathcal{G}^* \to \T\times L$ 
  of these two maps is hence a subtrivialization of $\mathcal{G}^*$ and we call 
  $(\mathcal{G}^*, L, \pi_L; \iota)$ the dual bundle of $(\mathcal{G}, L, q; \alpha)$. 
  (Note that $\mathcal{G}^*$ is, in general, neither locally compact nor 
  Hausdorff.)
  If, conversely, $(\mathcal{G}, L, q; \iota)$ is a group bundle with a $\T$-subtrivialization
  $\iota$,
  we set $\alpha_{\iota}\colon L\to \mathcal{G}^*$, $l\mapsto \iota_l$ and 
  call $(\mathcal{G}^*, L, \pi_L; \alpha_\iota)$ the dual bundle of 
  $(\mathcal{G}, L, q; \iota)$. We say that two group bundles with $\T$-subtrivializations
  $(\mathcal{G}, L, q; \iota)$ and $(\mathcal{G}', L', q'; \iota')$ 
  are isomorphic if their respective subtrivializations are, i.e., if there 
  is an isomorphism $(\Theta, \theta)\colon (\mathcal{G}, L, q) \to (\mathcal{G}', L', q')$
  such that the diagram 
  \begin{align*}
    \xymatrix@C=2.5cm{
      \T\times L \ar[r]^-{(z, l) \mapsto (z, \theta(l))} & \T\times L' \\
      \mathcal{G}' \ar[u]^-{\iota} \ar[r]^-{\Theta} & \mathcal{G}' \ar[u]^-{\iota'}
    }
  \end{align*}
  commutes.
\end{construction}

\begin{definition}\label{psb}
  Let $(K, \phi)$ be a topological dynamical system and $q\colon\ab K\to L$
  the projection onto its maximal trivial factor $L$. Then we define
  \begin{align*}
    \Sigma_\mathrm{p}(K, \phi) 
    \defeq \bigcup_{l\in L} \sigma_{\mathrm{p}}(T_{\phi_l})\times\{l\}
    \subset \C\times L.
  \end{align*}
  We denote the projection onto the second component by $\pi_L$ and equip
  $\Sigma_\mathrm{p}(K,\ab \phi)$ with the subspace topology induced by 
  $\C\times L$ if $\C$ carries the discrete topology. The bundle 
  $(\Sigma_\mathrm{p}(K, \phi), L, \pi_L)$ is then called the \emph{point spectrum
  bundle} of $(K, \phi)$. We say that the point spectrum bundles of two 
  systems are \emph{isomorphic} if there is an isomorphism of their canonical 
  subtrivializations.
  We say that the point spectrum bundles
  of two systems $(K, \phi)$ and $(M, \psi)$ are \emph{isomorphic} if there is
  a homeomorphism $\eta\colon L_\phi \to L_\psi$ such that
  \begin{align*}
    H\colon \Sigma_{\mathrm{p}}(K, \phi) \to \Sigma_{\mathrm{p}}(M, \psi), \quad
    (z, l) \mapsto (z, \eta(l))
  \end{align*}
  is a (well-defined) homeomorphism and call $(H, \eta)$ an isomorphism of the point 
  spectrum bundles.
\end{definition}

\begin{remark}\label{rem:psb}
  Let $(\mathcal{G}, L, \pi_L; \iota)$ be a group bundle with a $\T$-subtrivialization 
  $\iota\colon \mathcal{G}\to\T\times L$.
  Then $\iota$ induces an isomorphism 
  \begin{align*}
    (\mathcal{G}, L, \pi_L; \iota) \cong (\iota(\mathcal{G}), L, \pi_L; \id_{\iota(\mathcal{G})})
  \end{align*}
  and hence 
  \begin{align*}
    (\mathcal{G}, L, \pi_L; \iota)^* \cong (\iota(\mathcal{G})^*, L, \pi_L; (\id_{\iota(\mathcal{G}_l)})_{l\in L}).
  \end{align*}
  In particular, $\mathcal{G}$ and hence its dual are completely determined by 
  $\iota(\mathcal{G})$. Now, if $(\mathcal{G}, L, \pi_L;\iota)$ is the dual of a compact group 
  rotation bundle
  $(\mathcal{H}, L, p; \alpha)$ with minimal fibers and discrete spectrum, it 
  follows from the introduction to this section that 
  \begin{align*}
    \left(\iota(\mathcal{G})^*, L, \pi_L, (\id_{\iota(\mathcal{G}_l)})_{l\in L}\right)
    = \left(\Sigma_\mathrm{p}(\mathcal{H}, \phi_\alpha)^*, L, \pi_L, (\id_{\Sigma_{\mathrm{p}, l}(\mathcal{H}, 
\phi_\alpha)} )_{l\in L}\right).
  \end{align*}
  So we see that the dual bundle of a group rotation bundle with discrete spectrum and 
  minimal fibers is completely determined by 
  its point spectrum bundle.
\end{remark}

\begin{lemma}\label{lem:lsc}
  Let $(K,\phi)$ be a topological dynamical system with discrete spectrum.
  Then its point spectrum bundle is lower-semicontinuous.
\end{lemma}
\begin{proof}
  Suppose $(\lambda, l)\in\Sigma_{\mathrm{p}}(T_\phi)$ and let $f\in\mathrm{C}(K_l)$
  be a corresponding eigenfunction. Since $T_\phi$ has discrete
  spectrum, $\overline{\lambda}T_\phi$ is mean ergodic and since 
  all the eigenvalues of $T_{\phi_l}$ are simple, $\dim\fix(\overline{\lambda}T_\phi) = 1$.
  So as in the proof of \cref{thm:mergchar}, $f$ can be extended
  to a global fixed function
  $\tilde{f}\in\mathrm{C}(K)$ of $\overline{\lambda}T_\phi$. In particular,
  there is an open set $U\subset L$ such that for each 
  $l\in U$, $f_l\neq 0$ and $T_{\phi_l}(f_l) = \lambda f_l$.
\end{proof}

\begin{proposition}\label{pduality}
  Let $(\mathcal{G}, L, q; \alpha)$ be a compact group rotation bundle with 
  discrete spectrum and minimal fibers. Then it is isomorphic to 
  its  bi-dual bundle.
\end{proposition}
\begin{proof}
  The following diagram commutes:
  \begin{align*}
    \xymatrixcolsep{2cm}
    \xymatrixrowsep{1.5cm}
    \xymatrix{
      \left(\mathrm{E}(\mathcal{G}, \phi_\alpha)\times L, \phi_\alpha\right) \ar[d]^-\rho \ar[r]^-{(\psi, l)\mapsto 
(\delta_\psi, l)} & \left(\mathrm{E}(G, 
\phi_\alpha)^{**}\times L, \delta_{\phi_\alpha} \right) \ar[d]^-{\rho^{**}} \\
      \left(\mathcal{G}, \alpha\right) \ar[r]^-{g\mapsto \delta_g} & \left(\mathcal{G}^{**}, \delta_\alpha\right)
    }
  \end{align*}
  Since $\rho$ is a surjective, continuous map between compact spaces,
  $\mathcal{G}$ carries the final topology with respect to $\rho$. 
  This shows that the map $g\mapsto \delta_g$ is continuous and bijective.
  Combining \cref{rem:psb} and \cref{lem:lsc}, we see that 
  $(\mathcal{G}, L, q)^*$ is lower-semicontinuous and \cref{dualprops}
  shows that $\mathcal{G}^*$ embeds into $\mathrm{E}(\mathcal{G}, \phi_\alpha)\times L$
  and is therefore locally compact. By \cref{dualprops},
  $\mathcal{G}^{**}$ is thus Hausdorff. Since the map 
  $g\mapsto \delta_g$ is bijective, this implies
  $\mathcal{G}\cong \mathcal{G}^{**}$ and the claim follows.
\end{proof}

Here is now our final answer to the three aspects of the isomorphism problem
presented in the introduction.

\begin{theorem}\label{mthm}
  Let $(K, \phi)$ and $(M, \psi)$ be topological dynamical systems with 
  discrete spectrum and continuous sections of the canonical projections
  onto their respective maximal trivial factor.
  \begin{enumerate}[(a)]
    \item {\normalfont{(Representation)}} 
    The system $(K, \phi)$ is isomorphic to a compact group rotation bundle with minimal fibers.
    \item {\normalfont{(Uniqueness)}} The systems $(K, \phi)$ and $(M, \psi)$ are isomorphic if 
    and only if their point spectrum bundles are.
    \item {\normalfont{(Realization)}} The point spectrum bundle of $(K, \phi)$ is lower-semi\-con\-tin\-uous and 
  if $L$ is any compact space, 
  every lower-semicontinuous sub-group bundle of $(\T\times L, L, \pi_L)$ can be realized
  as the point spectrum bundle of a 
  topological dynamical system with discrete spectrum in the sense that the corresponding 
  canonical subtrivializations are isomorphic.
  \end{enumerate}
\end{theorem}
\begin{proof}
  The representation result is \cref{bundlechar}.
  Moreover, \cref{rem:psb} and \cref{pduality} show that the point spectrum bundle is a 
  complete isomorphism 
  invariant for compact group rotation bundles with minimal fibers and discrete spectrum
  and the representation theorem allows to extend this to $(K, \phi)$ and 
  $(M, \psi)$. The last part follows, analogously to the minimal case, from 
\cref{dualprops}\ref{item:bundlehausdorff}, 
  \cref{pduality} and \cref{rem:psb}.
\end{proof}

\begin{remark}
  Note that the statement of \cref{mthm} is false if 
  the continuous section assumption is removed. Indeed, one obtains a counterexample
  from \cref{nosection}.
\end{remark}

In order to obtain a similar result for measure-preserving 
systems, we first need to define their point spectrum bundles.
To motivate the definition, note that one could use the ergodic
decomposition to do this for separable systems. However, 
to treat non-separable systems, we base
our definition on topological models.

\begin{definition}\label{def:measpsb}
  Let $(\mathrm{X}, \phi)$ be a measure-preserving dynamical 
  system and take $(K, \mu_K; \psi)$ to be a topological model 
  corresponding to the algebra
  \begin{align*}
   \mathcal{A} \defeq \Kro_{\mathrm{L}^\infty}(T_\phi) =  
\operatorname{cl}_{\mathrm{L}^\infty}\bigcup_{|\lambda| = 1} 
    \ker_{\mathrm{L}^\infty}(\lambda I - T_\phi) \subset \mrL^\infty(\X).
  \end{align*}
  Let 
  $(\Sigma_{\mathrm{p}}(\X, \phi), L, p)$ be the point spectrum bundle of 
  $(K, \psi)$ and set $\nu\defeq p_*\mu_K$. We then call $(\Sigma(\X, \phi), L, p, \nu)$ the 
  \emph{point spectrum bundle} of $(\mathrm{X}, \phi)$. We say that the point 
  spectrum bundles of two systems $(\mathrm{X}, \phi)$ and $(\mathrm{Y}, \psi)$
  are \emph{isomorphic} if there is an isomorphism $(\Theta, \theta)\colon
  (\Sigma_{\mathrm{p}}(\X, \phi), L, p) \to (\Sigma_{\mathrm{p}}(\Y, \psi), L', p')$
  such that $\theta$ is measure-preserving.
\end{definition}


\begin{remark}\label{consistency}
Let $(K, \phi)$ be a topological dynamical system, 
$\mu$ a regular Borel measure on $K$, $q\colon K\to L$
the canonical projection onto the maximal trivial factor of $(K, \phi)$ and $\nu\defeq q_*\mu$. 
If the canonical map
$j\colon\Kro_{\mathrm{C}(K)}(T_\phi) \to \Kro_{\mathrm{L}^\infty(K, \mu)}(T_\phi)$
is an isomorphism, then $\Sigma_{\mathrm{p}}(K, \phi) 
= \Sigma_{\mathrm{p}}(K, \mu; \phi)$. This is in particular
the case for the group rotation bundles constructed in \cref{measurablebundle}.
\end{remark}

Recall that a regular Borel measure $\mu$ on a (hyper)stonean space $K$ is called
\emph{normal} if all rare sets are null-sets. If $\mu$ is a normal measure
on $K$ with full support, then the canonical embedding 
$\mathrm{C}(K) \hookrightarrow \mathrm{L}^\infty(K, \mu)$ is an isomorphism, cf. 
\cite[Corollary III.1.16]{Takesaki1979}. After this reminder, we can state the 
analogue of \cref{mthm} for measure-preserving systems.

\begin{theorem}\label{mainresult}
  Let $(\mathrm{X}, \phi)$ and $(\mathrm{Y}, \psi)$ be measure-preserving 
  dynamical systems with 
  discrete spectrum.
  \begin{enumerate}[(a)]
    \item \label{item:mr1} {\normalfont{(Representation)}} The system 
    $(\mathrm{X}, \phi)$ is Markov-isomorphic to a rotation 
    $(\mathcal{G}, \mu_{\mathcal{G}}; \phi_\alpha)$ on a compact 
    group rotation bundle with minimal fibers.
    \item \label{item:mr2} {\normalfont{(Uniqueness)}} The systems $(\mathrm{X}, \phi)$ and 
    $(\mathrm{Y}, \psi)$ are Markov-isomorphic 
    if and only if their point spectrum bundles are isomorphic.
    \item {\normalfont{(Realization)}} 
    The point spectrum bundle 
    of $(\mathrm{X}, \phi)$ is continuous. Conversely, if 
    $(L, \nu)$ is a hyperstonean compact 
    probability space such that $\nu$ is normal and $\supp \nu = L$ and $(\Sigma, L, p)$ is a 
    continuous sub-group bundle 
    of $(\T\times L, L, p)$ then $(\Sigma, L, p; \nu)$ can be realized as the point 
    spectrum bundle of a measure-preserving dynamical system with discrete spectrum.
  \end{enumerate}
\end{theorem}
\begin{proof}
  The representation result was proved in \cref{measurablebundle}. Using it, 
  the uniqueness can be reduced to the case of the special group rotation bundles  
  from \cref{measurablebundle} and for these, it follows from \cref{consistency}.
  Indeed, let $(\mathcal{G}, \mu_{\mathcal{G}};\phi_\alpha)$ and 
  $(\mathcal{H}, \mu_{\mathcal{H}};\phi_\beta)$ be two such rotations and 
  \begin{align*}
    (\Theta, \theta)\colon 
    \big(\Sigma_\mathrm{p}\left(\mathcal{G}, \mu_{\mathcal{G}};\phi_\alpha\right), 
    L, p, \nu \big)
    \to \big(\Sigma_{\mathrm{p}}\left(\mathcal{H}, \mu_{\mathcal{H}};\phi_\beta\right),
    M, q, \eta\big)
  \end{align*}
  an isomorphism of their point spectrum bundles.
  Then \cref{consistency} shows that $(\Theta, \theta)$ is, in particular, an 
  isomorphism of their topological point spectrum bundles and thus induces
  a (topological) isomorphism $(\Theta^*, \theta^{-1})$ of the corresponding dual bundles. 
  By \cref{pduality}, this yields an isomorphism 
  \begin{align*}
    (\Psi, \theta^{-1})\colon
    (\mathcal{H}, M, q; \phi_\beta) \to 
    (\mathcal{G}, L, p; \phi_\alpha).
  \end{align*}
  Using the disintegration formula from \cref{disintegration},
  one quickly checks that $\Psi_*\mu_{\mathcal{H}} = \mu_{\mathcal{G}}$ because 
  $\theta^{-1}_*\eta = \nu$.
  
    For part (c), let $(L, \nu)$ 
    be a hyperstonean compact 
    probability space such that $\nu$ is normal and $\supp \nu = L$ and let $(\Sigma, L, p)$ be a 
    continuous sub-group bundle 
    of $(\T\times L, L, p)$. Let $(\mathcal{G}, L, \pi_L, \phi_\alpha)$ be its dual
    group rotation bundle endowed with the measure
    \begin{align*}
      \mu_{\mathcal{G}} \defeq \int_L \dm_l \dnu.
    \end{align*}
    To prove that 
    $\Sigma = \Sigma_{\mathrm{p}}(\mathcal{G}, \mu_{\mathcal{G}}, \phi_\alpha)$, it
    suffices to show that each eigenfunction $f\in\mathrm{L}^\infty(\mathcal{G}, 
    \mu_{\mathcal{G}})$ has a representative $g\in\mathrm{C}(\mathcal{G})$ since 
    then 
    \begin{align*}
      \Sigma = \Sigma_{\mathrm{p}}(\mathcal{G}, \phi_\alpha) = \Sigma_{\mathrm{p}}(\mathcal{G}, 
    \mu_{\mathcal{G}}, \phi_\alpha)
    \end{align*}
    by \cref{pduality} and \cref{consistency}.
    So take 
    $[f]\in\mathrm{L}^\infty(\mathcal{G}, \mu_{\mathcal{G}})$
    with $T_\phi [f] = \lambda [f]$. Then $f_l = \lambda f_l$ for $\nu$-almost all 
    $l\in L$.
    Let $U_\lambda\subset L$ be the open subset of $l\in L$ such that 
    $(\lambda, l) \in\Sigma$ and note that 
    $U_\lambda$ is also closed since $(\Sigma, L, p)$ is upper-semicontinuous.
    
    Since $\mathcal{G}^*$ is 
    isomorphic to the point spectrum bundle $\Sigma_{\mathrm{p}}(\mathcal{G}, \phi_\alpha)$
    via an isomorphism $\Phi$
    by \cref{pduality}, the map $\eta\colon U\to \mathcal{G}^*$, $l \mapsto \Phi^{-1}(\lambda, l)$ 
is
    continuous. Extend $\eta$ to all of $L$ by setting $\eta(l)$ to the trivial character
    in $\mathcal{G}^*_l$ for $l\in L\setminus U$ and note that $\eta$ is continuous since 
    $U$ is open and closed.
    
    Now, for $l\in U$, each fiber $(\mathcal{G}_l, \phi_{\alpha, l})$ of $(\mathcal{G}, 
\phi_\alpha)$     is 
    a minimal group rotation and 
    so the eigenspace of the Koopman operator $T_{\phi_{\alpha, l}}$ corresponding to $\lambda$ 
    is at most one-dimensional 
    and therefore spanned by $\eta(l)\in\mathcal{G}_l^*$. So for $\nu$-almost every $l\in U$, there
    is a constant $c_l\in\C$ such that $f_l = c_l\eta(l)$ $\mathrm{m}_l$-almost everywhere.
    If we extend $c$ to $L$ by $0$, $[c]\in\mathrm{L}^\infty(L, \nu)$ since $[f]$ is in 
  $\mathrm{L}^\infty(\mathcal{G}, \mu)$. But 
    $\mathrm{C}(L) \cong \mathrm{L}^\infty(L, \nu)$ via the 
    canonical embedding and so we may assume that $c$ is continuous. If $q\colon\mathcal{G}\to 
L$ is the projection onto $L$, using \ref{item:first} of \cref{dualprops}, we see that the 
    function 
    $\tilde{f}\colon \mathcal{G}\to \C$, $x\mapsto c_{q(x)}\eta_{q(x)}(x)$ is in 
    $\mathrm{C}(\mathcal{G})$ and that $f = \tilde{f}$ $\mu$-almost everywhere by construction.
    
    Now let $(\X, \phi)$ be a measure-preserving dynamical system with discrete spectrum.
    In order to show that its point spectrum bundle is upper-semicontinuous, we may switch to its 
    representation $(\mathcal{G}, \mu_{\mathcal{G}}, \phi_\alpha)$ on a compact group rotation
    bundle $(\mathcal{G}, L, p;\phi_\alpha)$ constructed in \cref{measurablebundle}. Take 
    $\lambda\in \T$. By \cref{consistency}
    and \cref{lem:lsc}, the set 
    \begin{align*}
      U &\defeq \{ l\in L \mid (\lambda, l)\in \Sigma_{p}(\mathcal{G}, 
\mu_{\mathcal{G}}, \phi_\alpha)\} \\
       &= \{ l\in L \mid (\lambda, l)\in \Sigma_{p}(\mathcal{G},\phi_\alpha)\}
    \end{align*}
    is open. Via the isomorphism $\Theta\colon\Sigma_{p}(\mathcal{G},\phi_\alpha)\cong 
\mathcal{G}^*$,
    we see that the function $F\colon U \to \mathcal{G}^*$, $l \mapsto \Theta(\lambda, l)$
    selecting the (unique) character on $\mathcal{G}_l$ correponding to the eigenvalue $\lambda$
    is continuous. By \ref{item:first} of \cref{dualprops}, $F$ defines a continuous function 
    $f\colon p^{-1}(U) \to \C$ and we may extend $f$ to a measurable function on all of 
    $\mathcal{G}$ by 0. Then $T_\phi f = \lambda f$ and since the $\mathrm{C}(\mathcal{G})$- and 
    $\mathrm{L}^\infty(\mathcal{G},\mu_{\mathcal{G}})$-Kronecker space for $T_\phi$ are
    canonically isomorphic, 
    we can find a continuous representative $g\in\mathrm{C}(\mathcal{G})$ 
    for $f$. 
    
    Consider the following canonical isomorphisms:
    \begin{align*}
      &T_p\colon \mathrm{C}(L) \to 
\fix_{\mathrm{C}(\mathcal{G})}(T_{\phi_\alpha}) \hookrightarrow \mathrm{C}(\mathcal{G}), \\
      &T_p\colon \mathrm{L}^\infty(L, \nu) \to \fix_{\mathrm{L}^\infty}(T_{\phi_\alpha}) 
\hookrightarrow \mathrm{L}^\infty(\mathcal{G}, \mu_{\mathcal{G}}).
    \end{align*}
    Since $f$ is an eigenfunction
    $[|f|]\in\fix_{\mathrm{L}^\infty}(\mathcal{G},\mu_{\mathcal{G}})$
    and in fact, $|f| = \1_{p^{-1}(U)}$. Therefore, $T_p^{-1}([|f|]) = [\1_U]$.
    Moreover, $g$ is also an eigenfunction and so 
    $|g|\in\fix_{\mathrm{C}(\mathcal{G})}(T_{\phi_\alpha})$ and 
    $T_p^{-1}(|g|)\in \mathrm{C}(L)$. 
    
    But $|f| = |g|$ $\mu$-almost everyhwere and hence 
    $T_p^{-1}(|g|) = \1_U$ $\nu$-almost everywhere. But since 
    $\mathrm{C}(L) \cong \mathrm{L}^\infty(L, \nu)$ every equivalence class
    in $\mathrm{L}^\infty(L, \nu)$ contains precisely one continuous function,
    implying $T_p^{-1}(|g|) = \1_{\overline{U}}$.
    In particular, $g$ is an eigenfunction
    on $\mathcal{G}$ satisfying $g_l \neq 0$  for each $l\in\overline{U}$. Therefore, 
    $\overline{U} \subset U$ and hence $U = \overline{U}$. This shows that the 
    point spectrum bundle of $(\mathrm{X}, \phi)$ is upper-semicontinuous.
\end{proof}

\begin{remark}
  To conclude, let us briefly discuss how the different statements \ref{item:mr1} and 
  \ref{item:mr2} in \cref{mainresult} can be improved
  in the special case that $\X$ is a standard probability space:
  \begin{enumerate}[(a)]
    \item {\normalfont{(Representation)}} It is not difficult to see that if the measure space
    $\X$ is separable, the group rotation bundle can be chosen to be metrizable: Going back to
    the proof of \cref{measurablebundle}, the algebra $\mathcal{A}$ needs to be replaced by
    a separable subalgebra $\mathcal{B}$ which is still dense in $\mathrm{L}^1(\X)$. Using that $T_\phi$ 
    is mean ergodic on $\mathcal{A}$ and that there hence is a projection $P\colon \mathcal{A}\to 
\fix_{\mathcal{A}}(T_\phi)$, 
    this can be done in such a way that $\fix_{\mathcal{B}}(T_\phi)$ is generated
    by its characteristic functions. Therefore, its Gelfand representation space is totally 
    disconnected and using \cref{michael} instead of \cref{gleason}, one can continue the 
    proof of \cref{measurablebundle} analogously. By
    von Neumann's theorem \cite[Theorem 7.20]{EFHN2015} $(\mathrm{X}, \phi)$ is then not
    only Markov-isomorphic but point-isomorphic to the rotation on a compact group rotation bundle.
    \item {\normalfont{(Uniqueness)}} By von Neumann's theorem, Markov-isomorphy can be 
    replaced by point-isomorphy.
  \end{enumerate}
\end{remark}

\printbibliography

\end{document}